\documentclass[11pt, reqno]{article}
\usepackage{amsmath,amssymb,amsfonts,amsthm}
\usepackage{color}

\usepackage[colorlinks=true,linkcolor=blue,citecolor=blue,urlcolor=blue,pdfborder={0 0 0}]{hyperref}
\usepackage{cleveref}

\usepackage{caption}
\usepackage{thmtools}

\usepackage{mathrsfs}

\crefname{theorem}{Theorem}{Theorems}
\crefname{thm}{Theorem}{Theorems}
\crefname{lemma}{Lemma}{Lemmas}
\crefname{lem}{Lemma}{Lemmas}
\crefname{remark}{Remark}{Remarks}
\crefname{prop}{Proposition}{Propositions}
\crefname{defn}{Definition}{Definitions}
\crefname{corollary}{Corollary}{Corollaries}
\crefname{conjecture}{Conjecture}{Conjectures}
\crefname{question}{Question}{Questions}
\crefname{chapter}{Chapter}{Chapters}
\crefname{section}{Section}{Sections}
\crefname{figure}{Figure}{Figures}

\theoremstyle{plain}
\newtheorem{thm}{Theorem}[section]

\newtheorem{lemma}[thm]{Lemma}
\newtheorem{theorem}[thm]{Theorem}

\newtheorem{corollary}[thm]{Corollary}

\newtheorem{prop}[thm]{Proposition}
\newtheorem{conjecture}[thm]{Conjecture}

\theoremstyle{definition}

\theoremstyle{remark}
\newtheorem{remark}[thm]{Remark}

\numberwithin{equation}{section}

\renewcommand{\P}{\mathbb P}
\newcommand{\E}{\mathbb E}

\newcommand{\R}{\mathbb R}
\newcommand{\Z}{\mathbb Z}

\newcommand{\cF}{\mathcal F}
\newcommand{\cG}{\mathcal G}

\newcommand{\sA}{\mathscr A}
\newcommand{\sB}{\mathscr B}
\newcommand{\sC}{\mathscr C}
\newcommand{\sD}{\mathscr D}

\newcommand{\sF}{\mathscr F}
\newcommand{\sG}{\mathscr G}

\newcommand{\sO}{\mathscr O}

\newcommand{\sT}{\mathscr T}

\newcommand{\fG}{\mathfrak G}

\usepackage[margin=1.025in]{geometry}
\usepackage{extarrows}
\usepackage{titling}
\usepackage{commath}
\usepackage{mathtools}
\usepackage{titlesec}
\newcommand{\eps}{\varepsilon}
\usepackage{bbm}
\usepackage{setspace}
\usepackage{enumitem}
\usepackage{tikz}

\usepackage{cite}

\newcommand{\Aut}{\operatorname{Aut}}

\newcommand{\bP}{\mathbf P}
\newcommand{\bE}{\mathbf E}

\newcommand{\stab}{\operatorname{Stab}}

\newcommand{\opleq}{\preccurlyeq}

\renewcommand{\emptyset}{\varnothing}

\newcommand{\tn}{|\kern-.1em|\kern-0.1em|}

\newcommand{\gr}{\operatorname{gr}}

\newcommand\be{\begin{equation}}
\newcommand\ee{\end{equation}}

\def\eps{\varepsilon}

\newcommand{\sS}{\mathscr{S}}

\pretitle{\begin{flushleft}\Large}
\posttitle{\par\end{flushleft}\vskip 0.5em}
\preauthor{\begin{flushleft}}
\postauthor{
\\
\vspace{0.5em}
\footnotesize{Statslab, DPMMS, University of Cambridge.\\ Email: 
\href{mailto:t.hutchcroft@maths.cam.ac.uk}{t.hutchcroft@maths.cam.ac.uk}}
\end{flushleft}}
\predate{\begin{flushleft}}
\postdate{\par\end{flushleft}}
\title{\bf Locality of the critical probability for transitive graphs of exponential growth}

\renewenvironment{abstract}
 {\par\noindent\textbf{\abstractname.}\ \ignorespaces}
 {\par\medskip}

\titleformat{\section}
  {\normalfont\large \bf}{\thesection}{0.4em}{}

\author{{\bf Tom Hutchcroft}}
\begin{document}

\date{\small{\today}}

\maketitle

\setstretch{1.1}

\begin{abstract}
Around 2008, Schramm conjectured that the critical probabilities for Bernoulli bond percolation satisfy the following continuity property: If $(G_n)_{n\geq 1}$ is a sequence of transitive graphs converging locally to a transitive graph $G$ and $\limsup_{n\to\infty} p_c(G_n) < 1$, then $ p_c(G_n)\to p_c(G)$ as $n\to\infty$. We verify this conjecture under the additional hypothesis that there is a uniform exponential lower bound on the volume growth of the graphs in question. The result is new even in the case that the sequence of graphs is uniformly nonamenable.

In the unimodular case, this result is obtained as a corollary to the following theorem of independent interest: For every $g>1$ and $M<\infty$, there exist positive constants $C=C(g,M)$ and $\delta=\delta(g,M)$ such that if $G$ is a transitive unimodular graph with degree at most $M$ and growth $\operatorname{gr}(G) := \inf_{r\geq 1} |B(o,r)|^{1/r}\geq g$, then
\[
\bP_{p_c} \bigl(|K_o|\geq n\bigr) \leq C n^{-\delta}
\]
for every $n\geq 1$, where $K_o$ is the cluster of the root vertex $o$. The proof of this inequality makes use of new universal bounds on the probabilities of certain two-arm events, which hold for every unimodular transitive graph. 
\end{abstract}

\section{Introduction}\label{sec:intro}

Let $G=(V,E)$ be an infinite, connected, locally finite graph, and let $\omega_p \in \{0,1\}^E$ be \textbf{Bernoulli-$p$ bond percolation} on $G$, that is, the random subgraph of $G$ obtained by either deleting or retaining each edge of $G$ independently at random with retention probability $p\in [0,1]$. We will be particularly interested in the case that $G$ is \textbf{transitive}, i.e., that for any two vertices $u,v\in V$ there is an automorphism of $G$ mapping $u$ to $v$. Connected components of $\omega_p$ are called \textbf{clusters}. The \textbf{critical probability} is defined to be
\[
p_c(G) :=\inf \{p\in [0,1]: \omega_p \text{ has an infinite cluster almost surely}\}.
\]
Many features of percolation at and near $p_c(G)$ are expected to  depend only on the global, large-scale properties of $G$. For example, it has recently been shown that a transitive graph $G$ has $p_c(G)<1$ (i.e., percolation on $G$ undergoes a nontrivial phase transition) if and only if $G$ has superlinear volume growth \cite{1806.07733}, and it is conjectured that in this case $\omega_{p_c}$ does not have any infinite clusters almost surely. Moreover, it is conjectured that various features of critical percolation on $d$-dimensional Euclidean lattices are described by \emph{universal critical exponents} that depend on the dimension $d$ but not on the choice of $d$-dimensional lattice. 
%
 %
For detailed background on these questions and the progress that has been made on them, as well as percolation more generally, see e.g.\ \cite{grimmett2010percolation,LP:book,heydenreich2015progress,werner2007lectures}. 

In contrast to the predicted universal behaviour of percolation \emph{at} $p_c$, the \emph{value} of $p_c$ is highly lattice-dependent and is not determined by the large-scale properties of $G$. For example, the square and triangular lattices have very similar large-scale geometry but have $p_c=1/2$ and $p_c= 2\sin(\pi/18)$ respectively \cite[\S3.1]{grimmett2010percolation}. Around 2008, Oded Schramm conjectured that, \emph{subject to the global condition that $p_c$ is not too close to $1$}, the critical probability is not merely undetermined by the global geometry of the graph, but is in fact \emph{entirely determined by the local geometry of the graph} \cite[Conjecture 1.2]{MR2773031}. This conjecture has since emerged as one of the central questions in the study of percolation on transitive graphs, and is the primary subject of this paper.

Let us now  state this conjecture formally. A sequence of transitive graphs $G_n$ is said to \textbf{converge locally} to a transitive graph $G$ if for every $r\geq 1$ there exists $N<\infty$ such that for every $n\geq N$, every vertex $v_n$ of $G_n$, and every vertex $v$ of $G$, there exists a graph  isomorphism from the ball of radius $r$ around $v_n$ in $G_n$ to the ball of radius $r$ around $v$ in $G$ sending $v_n$ to $v$. 
In other words, $G_n$ converges locally to $G$ if the two graphs look the same within divergently large balls around the root. 

\begin{conjecture}[Schramm]
\label{conj:locality}
Let $(G_n)_{n\geq 1}$ be a sequence of transitive graphs converging locally to a transitive graph $G$, and suppose that $\limsup_{n\to\infty} p_c(G_n)<1$. Then $p_c(G_n)\to p_c(G)$ as $n\to\infty$.
\end{conjecture}

It may even be that the condition  $\limsup_{n\to\infty} p_c(G_n)<1$ can be replaced by the weaker condition that $p_c(G_n)<1$ for all $n$ sufficiently large.

In this paper, we verify the conjecture for graph sequences satisfying a uniform exponential lower bound on their volume growth. The result is new even in the case of uniformly nonamenable graphs sequences (i.e., for sequences satisfying $\limsup_{n\to\infty} \rho(G_n)<1$), which was raised as a case of particular interest in \cite{MR2773031}. In the unimodular case our proof also yields a quantitative estimate on the rate of convergence, see \cref{cor:rateofconvergence,remark:rateofconvergence}.

\begin{theorem}
\label{thm:locality}
Let $(G_n)_{n\geq 1}$ be a sequence of transitive graphs converging locally to a transitive graph $G$, and suppose that $\liminf_{n\to\infty} \operatorname{gr}(G_n)>1$. Then $p_c(G_n)\to p_c(G)$ as $n\to\infty$.
\end{theorem}

Here, we define the \textbf{growth} of a transitive graph $G$ to be $\gr(G)=\lim_{r\to\infty} |B(o,r)|^{1/r}$, where $o$ is a vertex of $G$.  The existence of this limit follows by submultiplicativity, and is clearly independent of the choice of $o$ by transitivity. A transitive graph $G$ is said to have \textbf{exponential growth} if $\gr(G)>1$. Similarly, the \textbf{spectral radius} of a graph $G$ is defined to be the exponential decay rate of the return probabilities $\rho(G)=\lim_{n\to\infty} p_{2n}(o,o)^{1/2n}$, where $p_{2n}(o,o)$ denotes the probability that a simple random walk on $G$ started at $o$ is at $o$ again at time $2n$. Similarly to above, this limit exists by supermultiplicativity and does not depend on the choice of root vertex. (This holds even without transitivity, see \cite[Proposition 6.6]{LP:book}.) The graph $G$ is said to be \textbf{nonamenable} if $\rho(G)<1$ and \textbf{amenable} otherwise. Every transitive nonamenable graph has exponential growth, but the converse does not hold. (In fact we have the quantitative bound $\gr(G) \geq \rho(G)^{-2}$, see \cite[Corollary 5.2]{MR986363}.)

We now briefly survey previous work on \cref{conj:locality}.
Pete \cite[\S14.2]{Pete} observed that the lower semi-continuity of $p_c$ (i.e., the statement that $\liminf_{n\to\infty} p_c(G_n) \geq p_c(G)$ whenever the graphs $(G_n)_{n\geq 1}$ are transitive and $G_n \to G$ locally) can easily be deduced from  the  mean-field lower bound on the percolation probability at $p_c+\eps$ \cite{MR864316,duminil2015new,aizenman1987sharpness}. Thus, to  prove \cref{conj:locality} it suffices to establish the upper semi-continuity estimate $\limsup_{n\to\infty}p_c(G_n)\leq p_c(G)$ under the hypothesis that $\limsup_{n\to\infty}p_c(G_n)<1$.  An alternative proof of lower semi-continuity is given in \cite[Page 4]{duminil2015new}. 
In their seminal paper \cite{MR1068308}, Grimmett and Marstrand studied percolation on \emph{slabs}
 of the form $\Z^{k} \times [0,n]^{d-k}$ in $\Z^d$, $d\geq 3$, and proved via a dynamical renormalization argument that
\[
p_c\bigl(\Z^{k} \times [0,n]^{d-k}\bigr)\xrightarrow[n\to\infty]{} p_c(\Z^d)
\]
whenever $2\leq k \leq d$. (Here $\Z^d$ and $\Z^k \times [0,n]^{d-k}$ are equipped with their usual hypercubic graph structure.) This theorem is of fundamental importance in the study of supercritical percolation in $\Z^d$ for $d\geq 3$. Although not an instance of \cref{conj:locality} since the slabs $\Z^{k} \times [0,n]^{d-k}$ are not transitive, the Grimmett-Marstrand Theorem does also imply the weaker statement that
\begin{equation}
\label{eq:SM}
p_c\bigl(\Z^{k} \times (\Z/n\Z)^{d-k}\bigr)\xrightarrow[n\to\infty]{} p_c(\Z^d)
\end{equation}
whenever $2 \leq k\leq d$, which is a special case of \cref{conj:locality}. This result was recently generalized by Martineau and Tassion \cite{MR3630298}, who proved \cref{conj:locality} in the special case that $G_n$ are all Cayley graphs of Abelian groups. 
 Closer to our setting, Benjamini, Nachmias, and Peres \cite{MR2773031} proved that \cref{conj:locality} holds for uniformly nonamenable graph sequences converging to a tree, while Song, Xiang, and Zhu \cite{song2014locality} showed that \cref{conj:locality} holds for uniformly nonamenable graph sequences in which every graph satisfies a certain very strong spherical symmetry property (in both of these settings the hypothesis $\limsup_{n\to\infty} p_c(G_n)<1$ holds automatically since $p_c \leq \gr^{-1} \leq \rho^2$ \cite{Hutchcroft2016944}).
Related questions of locality for other models such as self-avoiding walk \cite{benjamini2013euclidean,1412.0150,1704.05884} and the random-cluster model \cite{1707.07626} have also been studied.


The unimodular case of \cref{thm:locality} will be deduced as an immediate corollary of the following theorem, which gives quantitative control on the tail of the volume of critical clusters in transitive graphs of exponential growth and is of independent interest. 
(The nonunimodular case is handled via a separate argument which invokes the results of \cite{Hutchcroftnonunimodularperc}.) 
The proof of this theorem also yields a new proof that critical percolation on any transitive graph of exponential growth has no infinite clusters. We write $\bP_p$ for the law of $\omega_p$, write $K_v$ for the cluster\footnote{By abuse of notation we use $K_v$ to denote both the relevant subgraph of $\omega_p$ and the set of vertices it contains. The precise meaning should be clear from context.} of $v$ in $\omega_p$, and write $E(K_v)$ for the set of edges that \textbf{touch} $K_v$, i.e., have at least one endpoint in $K_v$.

\begin{theorem}
\label{thm:polybound}
For every $g>1$ and $M<\infty$ there exist constants $C=C(g,M)$ and $\delta=\delta(g,M)$ such that for every transitive unimodular graph $G$ with $\deg(o)\leq M$ and $\operatorname{gr}(G)\geq g$, the bound
\begin{equation}
\label{eq:polybound}
\bP_{p} (|E(K_o)|\geq n) \leq C n^{-\delta}
\end{equation}
holds for every $p\leq p_c$ and $n\geq 1$.
\end{theorem}

\begin{remark}
In order to apply \cref{thm:polybound} in the proof of \cref{thm:locality}, it is very important that all the constants depend only on the local geometry of the graph (here this dependence arises only through the degree) and the growth (which is the only aspect of the global geometry that we are assuming control of).
\end{remark}

Let us now discuss how \cref{thm:polybound} relates to previous results on critical percolation. 
It is conjectured that percolation on any transitive graph of exponential growth should have mean-field behaviour, so that in particular we should have that $\P(|K_o|\geq n) \preceq n^{-1/2}$ as we do on trees and on high-dimensional lattices. 
However, this conjecture is very far away from being proven, and even for nonamenable transitive graphs it was not previously known that the volume of the critical cluster satisfied any polynomial tail bound. 
Even in several of the special cases in which the conjecture has been proven to hold, the proofs do not give effective control of the constants that arise and therefore cannot be used to prove locality \cite{Hutchcroftnonunimodularperc,1804.10191}. The one case in which it is known how to prove this conjecture in a sufficiently quantitative way that locality can be deduced is under perturbative assumptions, i.e. that the graph is not just nonamenable but either \emph{highly nonamenable} \cite{MR1833805,MR1756965} or of \emph{high girth} \cite{MR3005730}. In particular, it can be deduced from the techniques of \cite{MR1756965,MR3005730} that if $G_n \to G$ locally and $\limsup_{n\to\infty} \rho(G_n)<1/2$ then $p_c(G_n)\to p_c(G)$.

 In their seminal paper \cite{BLPS99b}, Benjamini, Lyons, Peres, and Schramm proved that critical percolation on any nonamenable unimodular transitive graph does not have any infinite clusters a.s. However, their proof is ineffective in the sense that it cannot be used to establish any explicit bounds on the tail of the volume of critical clusters. Later, Tim\'ar \cite{timar2006percolation} proved that critical percolation on any \emph{nonunimodular} transitive graph does not have \emph{infinitely many} infinite clusters a.s., again via an ineffective argument. More recently, we proved that critical percolation on any transitive graph of exponential growth cannot have a \emph{unique} infinite cluster \cite{Hutchcroft2016944} a.s., which together with the aforementioned results of \cite{BLPS99b,timar2006percolation} implied that critical percolation on any transitive graph of exponential growth
does not have any infinite clusters a.s. The proof of \cite{Hutchcroft2016944} also established that the quantitative bound
\begin{equation}
\label{eq:kappabound}
\kappa_{p}(n):=\inf\bigl\{\tau_{p}(u,v) : d(u,v)\leq n\bigr\} \leq \gr(G)^{-n}
\end{equation}
holds for every $p\leq p_c$, 
where $\tau_p(u,v)$ denotes the probability that $u$ and $v$ are connected in $\omega_p$.
This bound plays an important role in the proofs of the main theorems of this paper.
 However, the rest of the proof given in \cite{Hutchcroft2016944} that all critical clusters are finite is once again ineffective and does not give any control of the tail of the volume of critical clusters. (The inequality \eqref{eq:kappabound} does directly imply the special case of \cref{thm:locality} in which the limit graph $G$ is amenable. The reader may find it an enlightening exercise to prove this.)

\medskip

\noindent \textbf{Proof sketch.} 
To prove \cref{thm:polybound}, we develop a general method of converting two-point function bounds such as \eqref{eq:kappabound} into volume-tail bounds such as \eqref{eq:polybound}. To do this, we apply a variation on the uniqueness proof of Aizenman, Kesten, and Newman \cite{MR901151} to establish a universal bound on the probability of the two-arm type event $\sS_{e,n}$ that the edge $e$ is closed and that its endpoints are in distinct finite clusters each of which touches at least $n$ edges (\cref{cor:twoarmunimod}). 
 This bound holds for \emph{every} unimodular transitive graph and every $p\in (0,1]$, and is discussed in detail in the next subsection. 
We then apply a surgery argument using the Harris-FKG inequality and insertion-tolerance to get that
\[
p^k \left[ \bP_p(|E(K_o)|\geq n)^2 -\kappa_p(k)\right] \leq \frac{p}{1-p}\sup_{e\in E}\bP_p(\sS_{e,n})
\]
for every $0<p<p_c$ and $n,k\geq 1$.
\cref{thm:polybound} then follows from \eqref{eq:kappabound} and \cref{cor:twoarmunimod} by direct calculation.

\begin{remark}
\label{remark:exponent}
The proof of \cref{thm:polybound} yields the simple explicit bound
\begin{equation}
\bP_{p_c} \bigl(|E(K_o)|\geq n\bigr) \leq \sqrt{2} \left( 66 d \left[ \frac{{1}}{(1-{p_c})n} \right]^{1/2} \right)^{1/2\alpha}  \leq \frac{5 d^{1/4}}{(1-\gr^{-1})^{1/8}} n^{-\beta}
\end{equation}
where $\alpha =1 -\log p_c (G) /\log \gr(G) \geq 2$ and $\beta = 1/4(1+ \log(d-1) /\log \gr)$.
With further work it is possible to use our methods to get an estimate of the form
\[
\bP_{p_c} \bigl(|E(K_o)|\geq n\bigr) \leq n^{-1/(4\alpha-2)+o(1)},
\]
see \cref{remark:improvedpolybound}. 
It seems that a new idea is required to improve the exponent any further than this.
 Since $p_c \leq \gr^{-1}$ for every transitive graph, the best exponent that these bounds can ever give is $1/6$, and in general the exponent we obtain can be much worse than this. 
\end{remark}


\subsection{The two-ghost inequality}

As discussed above, in the unimodular case, a central input to the proofs of our main theorems is a universal bound on the probabilities of certain two arm-events that holds for every unimodular transitive graph. We call this bound the \emph{two-ghost inequality}. Our proof of this bound was inspired by the work of Aizenman, Kesten, and Newman \cite{MR901151}, who used a similar method to prove that percolation on $\Z^d$ has at most one infinite cluster almost surely. See \cite{MR929129} for a simplified exposition of their proof. Roughly speaking, their proof uses an ingenious summation-exchange argument to rewrite the probability of a certain two-arm event $\sA$ in terms of an expectation roughly of the form $\E[T^{-1}Z_T \mathbbm{1}(T<\infty) \mathbbm{1}(\sB)]$, where $(Z_n)_{n\geq 0}$ is a martingale with bounded, i.i.d.\ increments, $T$ is a stopping time, and $\sB$ is a certain one-arm event for the percolation configuration. On the event $\sB$ the stopping time $T$ must be large, and one can therefore easily bound this expectation using e.g.\ Doob's $L^2$ maximal inequality to obtain that the probability of the two-arm event is small as desired. (They phrase their argument somewhat differently than this, using large-deviation estimates rather than martingale techniques, but the core idea of their proof is as above.) 

The proof of \cite{MR901151} also yields the following quantitative estimate for percolation on $\Z^d$. Let $e$ be an edge of $\Z^d$ and let $\sA_{n}$ be the event that $e$ is closed and that the two endpoints of $e$ are in distinct clusters each of which has diameter at least $n$. Then for every $p\in (0,1)$ there exists a constant $C_{d,p}$ such that
\begin{equation}
\label{eq:AKNoriginal}
\bP_p(\sA_n) \leq C_{d,p} \frac{\log n}{\sqrt{n}}.
\end{equation}
 See \cite{MR3395466} for a discussion of and improvement to this bound. 
Although it is possible to adapt the Aizenman-Kesten-Newman argument to prove uniqueness on any amenable transitive graph, the  bound one  obtains on $\bP_p(\sA_n)$ becomes increasingly poor as the isoperimetry improves, 
 and we do not obtain \emph{any} information about percolation on \emph{nonamenable} graphs.

In this paper, we prove a variation on \eqref{eq:AKNoriginal} that applies universally to every unimodular transitive graph, and improves on the bound \eqref{eq:AKNoriginal} even in the case of $\Z^d$.  Aside from this universality, the most significant differences between our inequality and \eqref{eq:AKNoriginal} are as follows: Firstly, we work with two-arm events in which \emph{at least one of the two clusters is finite}, so that in particular our inequality does not directly\footnote{In fact, it is also possible to prove uniqueness using an extension of our techniques, and this proof also yields interesting quantitative information about supercritical percolation on amenable transitive graphs. See \cref{remark:improvedtwoghost} for a discussion.} imply uniqueness of the infinite cluster as \eqref{eq:AKNoriginal} does.  Such a modification is of course necessary in order to obtain an inequality that is valid in the nonamenable setting. Secondly, rather than studying the \emph{diameter} of clusters, we study their \emph{volume}. This is done somewhat indirectly by introducing a \emph{ghost field} as in \cite{aizenman1987sharpness}. This modification allows us to work directly in infinite volume rather than working in finite volume and taking limits as in  \cite{MR901151}, and also leads to stronger results since the volume is an upper bound on the diameter. In particular, we apply the mass-transport principle to carry out the summation-exchange argument of \cite{MR901151} directly in infinite volume. This is where the assumption of unimodularity is required.

We now introduce the ghost field and state the two-ghost inequality.
  Let $G=(V,E)$ be a connected, locally finite graph, and let $p\in (0,1)$ and $h>0$. Let $\omega_p\in \{0,1\}^E$ be Bernoulli-$p$ bond percolation on $G$. Independently of $\omega_p$, let $\cG_h \in \{0,1\}^E$  be a random subset of $E$ where each edge $e$ of $E$ is included in $\cG_h$ independently at random, with probability $1-e^{-h}$ of being included. (It is more standard to put the ghosts on the vertices, but putting them on the edges is convenient for the calculations we do here.) Following \cite{aizenman1987sharpness}, we call $\cG_h$ the \textbf{ghost field} and call an edge \textbf{green} if it is included in $\cG_h$. We write $\bP_{p,h}$ for the joint law of $\omega_p$ and $\cG_h$, and define $\sT_e$ to be the event that $e$ is closed and that the endpoints of $e$ are in distinct clusters, each of which touches some green edge, and at least one of which is finite.

\begin{theorem}
\label{thm:twoghostunimod}
 Let $G$ be a unimodular transitive graph of degree $d$. Then 
\begin{equation}
\label{eq:twoghostunimod}
\bP_{p,h}\bigl(\sT_{e}\bigr) \leq 33 \cdot d\left[\frac{1-p}{p}h\right]^{1/2}
\end{equation}
for every $e\in E$, $p\in (0,1]$ and $h>0$.
\end{theorem}

The bound of \cref{thm:twoghostunimod} can easily be converted into a bound on a two-arm type event that does not refer to the ghost field. 
Let $\sS_{e,n}$ be the event that $e$ is closed and that the endpoints of $e$ are in distinct clusters, each of which touches at least $n$ edges, and at least one of which is finite.

\begin{corollary}
\label{cor:twoarmunimod}
 Let $G$ be a unimodular transitive graph of degree $d$. Then 
\begin{equation}
\label{eq:twobodyunimod}
\bP_p(\sS_{e,n})\leq 66 \cdot d\left[\frac{1-p}{p n}\right]^{1/2}
\end{equation}
for every $e\in E$, $p\in(0,1]$ and $n\geq 1$.
\end{corollary}

\begin{remark}
The factor of $d$ on the right of \eqref{eq:twoghostunimod} and \eqref{eq:twobodyunimod} can be replaced by the reciprocal of the probability that an edge chosen uniformly at random from those incident to $o$ is in the same orbit under $\Aut(G)$ as $e$, and in particular can be removed entirely on edge-transitive graphs such as $\Z^d$. Further improvements to and variations on these bound are discussed in  \cref{remark:improvedtwoghost}.
\end{remark}




\section{Unimodularity, nonunimodularity, and the mass-transport principle}
\label{section:unimodbackground}

We now briefly review the notions of unimodularity, nonunimodularity, and the mass-transport principle, referring the reader to e.g.\ \cite[Chapter 8]{LP:book} for further background.

Let $G=(V,E)$ be a transitive graph, and let $\Aut(G)$ be the group of automorphisms of $G$. We define the \textbf{modular function} $\Delta=\Delta_G:V^2 \to (0,\infty)$ to be
\[
\Delta(u,v) = \frac{|\stab_v u|}{|\stab_u v|},
\]
where $\stab_v = \{\gamma \in \Aut(G): \gamma v = v\}$ is the stabilizer of $v$ in $\Aut(G)$ and $\stab_v u = \{ \gamma u : \gamma \in \stab_v\}$ is the orbit of $u$ under $\stab_v$. 
We say that $G$ is \textbf{unimodular} if $\Delta(u,v)=1$ for every $u,v\in V$, and \textbf{nonunimodular} otherwise. Every Cayley graph is unimodular, as is every amenable transitive graph \cite{MR1082868}.

Let $G$ be a unimodular transitive graph. The \textbf{mass-transport principle} states that if $F:V^2\to [0,\infty]$ is diagonally-invariant in the sense that $F(\gamma u, \gamma v)=F(u,v)$ for every $u,v \in V$, then
\begin{equation}
\label{eq:MTP}
\sum_{v\in V} F(o,v) = \sum_{v\in V} F(v,o).
\end{equation}
It will be convenient for us to use the following variation on the mass-transport principle. 
Given a graph $G=(V,E)$, write $E^\rightarrow$ for the set of oriented edges of $G$, where an oriented edge $e$ is oriented from its tail $e^-$ to its head $e^+$.
Let $G$ be a transitive graph, let $o$ be an arbitrarily chosen root vertex of $G$, and let $\eta$ be chosen uniformly at random from the set of oriented edges of $G$ emanating from $o$. Then for every $F:E^\rightarrow \times E^\rightarrow \to [0,\infty]$ that is diagonally-invariant in the sense that $F(\gamma e_1, \gamma e_2)=F(e_1,e_2)$ for every $e_1,e_2\in E^\rightarrow$ and $\gamma \in \Aut(G)$, we have that
\begin{equation}
\E\sum_{e\in E^\rightarrow} F(\eta,e) = \E\sum_{e\in E^\rightarrow} F(e,\eta),
\label{eq:edgeMTP}
\end{equation}
where the expectation is taken over the random root edge $\eta$. 
This
 equality is 
 easily seen to follow by applying \eqref{eq:MTP} 
 to the function $\tilde F : V^2 \to [0,\infty]$ defined by setting $\tilde F(u,v) = \sum_{e_1^-=u} \sum_{e_2^-=v} F(e_1,e_2)$ for each $u,v\in V$. The equality \eqref{eq:edgeMTP} also holds for signed diagonally-invariant functions $F:E^\rightarrow\times E^\rightarrow \to\R$ satisfying the integrability condition
\begin{equation}
\label{eq:integrability}
\E\sum_{e\in E^\rightarrow}|F(\eta,e)|<\infty.
\end{equation}
This can be seen by applying \eqref{eq:edgeMTP} separately to the positive and negative parts of $F$, defined by
 $F^+(e_1,e_2)=0\vee F(e_1,e_2)$ and $F^-(e_1,e_2)=  0\vee (-F(e_1,e_2))$.

\begin{remark}
The formulation of the mass-transport principle given in \eqref{eq:edgeMTP} holds more generally for \emph{reversible random rooted graphs}. Such a random rooted graph can be obtained from a \emph{unimodular random rooted graph} of finite expected degree by biasing by the degree. See e.g.\ \cite{AL07,CurienNotes} for background on these notions.
\end{remark}

\section{Proof of the two-ghost inequality}

In this section we prove \cref{thm:twoghostunimod,cor:twoarmunimod}.
Given a graph $G$, $p\in [0,1]$, and a finite subgraph $H$ of $G$, we write 
\[h_p(H)= p |\partial H| - (1-p)|E_\circ(H)|,\]
where $\partial H$ denotes the set of (unoriented) edges of $G$ that touch the vertex set of $H$ but are not included in $H$, and $E_\circ(H)$ denotes the set of (unoriented) edges of $G$ that are included in $H$. We also write $E(H)$ for the set of (unoriented) edges of $G$ that touch the vertex set of $H$.

Let $G$ be a connected, locally finite graph $G$, and let $o$ be a vertex of $G$. 
We write $\P_{p,h}$ and $\E_{p,h}$ for probabilities and expectations taken with respect to the law of the percolation configuration $\omega_p$, the independent ghost field $\cG_h$, and the independent choice of a uniformly random oriented edge $\eta$ emanating from $o$. 

\begin{lemma}
\label{lem:AKN} Let $G$ be a connected, locally finite, unimodular transitive graph. Then the inequality
\begin{align*}
\P_{p,h}(\sT_\eta) \leq \frac{2}{p}\bE_{p,h}\left[\frac{|h_p(K_{o})|}{|E(K_{o})|} \mathbbm{1}\bigl(|K_{o}| < \infty \text{ and } E(K_{o}) \cap \cG_h \neq \emptyset \bigr)\right]
\end{align*}
holds for every $p\in (0,1]$ and $h>0$.
\end{lemma}

Note that for every edge $e$ of $G$, we have that the unoriented edge obtained by forgetting the orientation of $\eta$ has probability at least $1/\deg(o)$ to be in the same $\Aut(G)$ orbit as $e$. Thus, \cref{lem:AKN} immediately implies that if $G$ is unimodular then
\begin{align}
\label{eq:twoghostdegree}
\bP_{p,h}(\sT_e) \leq \frac{2 \deg(o)}{p}\bE_{p,h}\left[\frac{|h_p(K_{o})|}{|E(K_{o})|} \mathbbm{1}\bigl(|K_{o}| < \infty \text{ and } E(K_{o}) \cap \cG_h \neq \emptyset \bigr)\right]
\end{align}
for every $e\in E$, $p\in (0,1]$ and $h>0$.

\begin{proof}[Proof of \cref{lem:AKN}]
Let $\sT_e'$ be the event that the endpoints of $e$ are in distinct \emph{finite} clusters each of which touches $\cG_h$. 
Let $\sG_e$ be the event that there exists a finite cluster touching $e$ and $\cG_h$. Observe that for each edge $e$ of $G$ we have the equality 
\begin{equation*}
\mathbbm{1}(\sT_e') = \mathbbm{1}(\omega(e)=0)\cdot \#\{\text{finite clusters touching $e$ and $\cG_h$}\}
\\- \mathbbm{1}\bigl(\{\omega(e)=0\}\cap \sG_e), 
\end{equation*}
and hence
\begin{multline}
\label{eq:unimodghost1}
\bP_{p,h}(\sT_e') = \bE_{p,h}\left[\mathbbm{1}(\omega(e)=0)\cdot\#\{\text{finite clusters touching $e$ and $\cG_h$}\}\right]
\\- \bP_{p,h}\bigl(\{\omega(e)=0\} \cap \sG_e\bigr). 
\end{multline}
Let $\sF_e$ be the event that every cluster touching $e$ is finite. Observe that the event $\sF_e \cap \sG_e$ is independent of the value of $\omega(e)$, and consequently that
\begin{multline}
\bP_{p,h}\bigl(\{\omega(e)=0\} \cap \sF_e \cap \sG_e\bigr)
= \frac{1-p}{p} \bP_{p,h}\bigl(\{\omega(e)=1\} \cap \sF_e \cap \sG_e \bigr).
\\
= \frac{1-p}{p} \bP_{p,h}\bigl(\{\omega(e)=1\}  \cap \sG_e\bigr).
\label{eq:unimodghost2}
\end{multline}
Combining \eqref{eq:unimodghost1} and \eqref{eq:unimodghost2} yields that
\begin{multline}
\label{eq:AKNmainstep}
\bP_{p,h}(\sT_e') = \bE_{p,h}\left[\mathbbm{1}(\omega(e)=0)\cdot \#\{\text{finite clusters touching $e$ and  $\cG_h$}\}\right]
\\- \frac{1-p}{p}\bP_{p,h}(\{\omega(e)=1\} \cap \sG_e) - \bP_{p,h}\bigl(\{\omega(e)=0\} \cap \sG_e \setminus \sF_e\bigr).
\end{multline}
Finally, observe that $\{\omega(e)=0\} \cap \sG_e \setminus \sF_e$ and $\sT_e'$ are disjoint and that the events $\sT_e$ and $\sT_e' \cup (\{\omega(e)=0\} \cap \sG_e \setminus \sF_e)$ coincide up to a null set, so that \eqref{eq:AKNmainstep} implies that
\begin{multline}
\label{eq:AKNmainstep2}
\bP_{p,h}(\sT_e) = \bE_{p,h}\left[\mathbbm{1}(\omega(e)=0)\cdot \#\{\text{finite clusters touching $e$ and  $\cG_h$}\}\right]
\\- \frac{1-p}{p}\bP_{p,h}(\{\omega(e)=1\} \cap \sG_e).
\end{multline}

We will now apply the assumption that $G$ is unimodular. 
Define a mass-transport function $F:E^\rightarrow\times E^\rightarrow \to \R$ by
\begin{equation*}
F(e_1,e_2) = \bE_{p,h}\sum\left\{\frac{\mathbbm{1}(\omega(e_1)=0) - \frac{1-p}{p}\mathbbm{1}(\omega(e_1)=1)}{2|E(K)|} : \begin{array}{l}\text{$K$ is a finite cluster of $\omega$}\\ \text{touching $e_1,e_2$, and $\cG_h$}\end{array}\right\},
\end{equation*}
where we write $\sum\{x(i) :i\in I\} = \sum_{i\in I} x(i)$. The factor of $2$ accounts for the fact that each edge in $E(K)$ can be oriented in two directions. 
Note that the multiset of numbers being summed over has size either $0,1,$ or $2$. Thus, we easily verify the integrability condition
$\E\sum_{e\in E^\rightarrow} |F(\eta,e)| \leq 2\left[ 1 + (1-p)/p\right]<\infty$,
and applying the mass-transport principle \eqref{eq:edgeMTP} we obtain that
\begin{multline}
\label{eq:AKN_MTP}
\P_{p,h}(\sT_\eta) = \E \sum_{e\in E^\rightarrow}F(\eta,e) = \E \sum_{e\in E^\rightarrow}F(e,\eta)   \\=
\frac{1}{p}\E_{p,h} \sum\left\{\frac{h_p(K)}{|E(K)|} : \begin{array}{l}\text{$K$ is a finite cluster of $\omega$}\\ \text{touching $\eta$ and $\cG_h$}\end{array}\right\},
\end{multline}
where the first equality follows from \eqref{eq:AKNmainstep2}. 
For each vertex $v$ of $G$, define 
$\sO_v$ to be the event that $K_v$ is finite and touches $\cG_h$. 
We deduce immediately from \eqref{eq:AKN_MTP} that
\begin{multline*}\P_{p,h}(\sT_\eta) \leq \frac{1}{p}\E_{p,h}\sum\left\{\frac{|h_p(K)|}{|E(K)|} : \begin{array}{l}\text{$K$ is a finite cluster of $\omega$}\\ \text{touching $\eta$ and $\cG_h$}\end{array}\right\}\\
\leq\frac{1}{p}\E_{p,h}\left[\frac{|h_p(K_{\eta^-})|}{|E(K_{\eta^-})|}\mathbbm{1}\bigl(\sO_{\eta^-} \bigr)+ \frac{|h_p(K_{\eta^+})|}{|E(K_{\eta^+})|}\mathbbm{1}\bigl(\sO_{\eta^+} \bigr)\right]
=\frac{2}{p}\bE_{p,h}\left[\frac{|h_p(K_{o})|}{|E(K_{o})|}\mathbbm{1}\bigl(\sO_{o} \bigr)\right]
\end{multline*}
as claimed, where we have used transitivity in the final equality.
\end{proof}

\begin{proof}[Proof of \cref{thm:twoghostunimod}]
In light of \cref{lem:AKN} and the estimate \eqref{eq:twoghostdegree} that follows from it, it suffices to prove that for every connected, locally finite graph $G$ and every vertex $v$ of $G$, we have that
\begin{align*}
\bE_{p,h}\left[\frac{|h_p(K_{v})|}{|E(K_{v})|} \mathbbm{1}\bigl(|K_{v}| < \infty,\, E(K_{v}) \cap \cG_h \neq \emptyset \bigr)\right] &= \bE_{p}\left[\frac{|h_p(K_{v})|}{|E(K_{v})|} (1-e^{-h |E(K_v)|}) \mathbbm{1}\bigl(|K_v|<\infty\bigr)\right]\\ &\leq \frac{33}{2} \sqrt{p(1-p) h}.
\end{align*}

Consider the following procedure for exploring the cluster of $v$. 
Fix a well-ordering $\opleq$ of the edges of $G$. 
At each stage of the exploration we will have a set of vertices $U_n$, a set of revealed open edges $O_n$, and a set of revealed closed edges $C_n$. 
We begin by setting $U_0=\{v\}$ and  $C_0=O_0=\emptyset$. Let $n\geq 1$. Given what has happened up to and including step $n-1$ of the exploration, we define $(U_n,O_n,C_n)$ as follows: If every edge touching $U_{n-1}$ is included in $O_{n-1}\cup C_{n-1}$, we set $(U_n,O_n,C_n)=(U_{n-1},O_{n-1},C_{n-1})$. Otherwise, we take $E_n$ to be the $\opleq$-minimal element of the set of edges that touch $U_{n-1}$ but are not in $O_{n-1}$ or $C_{n-1}$. If $\omega_p(E_n)=1$, we set $O_n=O_{n-1}\cup\{E_n\}$, $C_n=C_{n-1}$, and set $U_{n}$ to be the union of $U_n$ with the set of endpoints of $E_n$. Otherwise, $\omega_p(E_n)=0$ and we set $O_n=O_{n-1}$, $C_n =C_{n-1} \cup \{E_n\}$, and $U_n = U_{n-1}$. Let $(\cF_n)_{n\geq 0}$ be the filtration generated by this exploration process. 

Let $T$ be the first time $n$ that every edge touching $U_n$ is included in $O_n \cup C_n$, setting $T=\infty$ if this never occurs.
It is easily verified that $(U_T,O_T,C_T,T)$ is equal to $(K_v,E_\circ(K_v),\partial K_v,|E(K_v)|)$. 
Let $(Z_n)_{n\geq 0}$ be defined by $Z_0=0$ and
\[Z_n = \sum_{i=1}^{n\wedge T} \left[(1-p)\mathbbm{1}(\omega_p(E_i)=1) - p \mathbbm{1}(\omega_p(E_i)=0)\right].\]
Then we have that $Z_n=(1-p)|O_n| -p |C_n|$ for every $n\leq T$, and consequently that 
\begin{equation}
\label{eq:converttoMTG}
\bE_{p}\left[\frac{|h_p(K_{v})|}{|E(K_{v})|} (1-e^{-h |E(K_v)|}) \mathbbm{1}\bigl(|K_v|<\infty\bigr)\right] = \bE_p\left[ \frac{|Z_T|}{T}\bigl(1-e^{-hT}\bigr)\mathbbm{1}(T<\infty) \right].
\end{equation}
%
We will control this expectation using the elementary bound
\begin{multline}
  \bE_p\left[ \frac{|Z_T|}{T}\bigl(1-e^{-hT}\bigr)\mathbbm{1}(T<\infty) \right]
\leq \sum_{k\geq 0} \frac{1-e^{-h2^{k}}}{2^k} \bE_p\left[ \max_{2^k \leq n \leq 2^{k+1}} |Z_n| \mathbbm{1}(2^k \leq T \leq 2^{k+1})\right],
\label{eq:scales}
\end{multline}
where we used the fact that $(1-e^{-hx})/x$ is a decreasing function of $x>0$. 
 The process $(Z_n)_{n\geq 0}$ is a martingale with respect to $(\cF_n)_{n\geq 1}$ and by orthogonality of martingale increments has \[\bE_p Z_n^2 = \sum_{i=1}^n \bE_p (Z_i-Z_{i-1})^2 = p(1-p)\sum_{i=1}^n \bP_p(T\geq i) = p(1-p) \bE_p[ T \wedge n] \leq p(1-p)n\] for every $n\geq 0$.
Thus, we may apply Doob's $L^2$ maximal inequality to deduce that
\[
\bE_p\left[ \max_{2^k \leq n \leq 2^{k+1}} Z_n^2 \right] \leq \bE_p\left[ \max_{0 \leq n \leq 2^{k+1}} Z_n^2 \right] \leq 4 \bE_p \left[ Z_{2^{k+1}}^2\right] \leq 8p(1-p) 2^{k},
\]
and hence by \eqref{eq:scales} that
\begin{align*}
  \bE_p\left[ \frac{|Z_T|}{T}\bigl(1-e^{-hT}\bigr)\mathbbm{1}(T<\infty) \right]
\leq \sqrt{8p(1-p)}\sum_{k\geq 0} \frac{1-e^{-h2^{k}}}{2^k} 2^{k/2}.
\end{align*}
Next, we use the bound $1-e^{-h2^{k}} \leq h2^{k}$ when $k\leq \lfloor \log_2 (1/h) \rfloor$ and the bound $1-e^{-h2^{k}} \leq 1$ when $k\geq \lfloor \log_2 (1/h) \rfloor+1$ to obtain that
\begin{align*}
\bE_p\left[ \frac{|Z_T|}{T}\bigl(1-e^{-hT}\bigr)\mathbbm{1}(T<\infty) \right] 
&\leq \sqrt{8p(1-p)}\left[\sum_{k= 0}^{\lfloor \log_2 (1/h) \rfloor} h 2^{k/2} + \sum_{k \geq \lfloor \log_2 (1/h)\rfloor +1} 2^{-k/2} \right]\\
&\leq \frac{(\sqrt{2}+1)\sqrt{8p(1-p)h}}{\sqrt{2}-1} \leq \frac{33}{2} \sqrt{p(1-p)h}, 
\end{align*}
where we used the bound $(\sqrt{2}+1) \sqrt{8} / (\sqrt{2} -1) =16.485\ldots \leq 33/2$ to simplify the expression. In light of \eqref{eq:converttoMTG}, the proof is completed.
\end{proof}

\begin{proof}[Proof of \cref{cor:twoarmunimod}]
Let $\sD_e$ be the event that $e^+$ and $e^-$ are in distinct clusters at least one of which is finite. 
By positive association of the Bernoulli process $\cG_h$, we have that
\begin{multline*}
\mathbf{P}_{p,h}(\sT_e) \geq \bE_{p}\left[\Bigl(1-e^{-h|E(K_{e^-})|}\Bigr)\Bigl(1-e^{-h|E(K_{e^+})|}\Bigr)\mathbbm{1}\bigl(\sD_e\bigr) \right] \\\geq (1-e^{-hn})^2\bP_{p}\bigl(\sD_e \cap \{|E(K_{e^-})|,|E(K_{e^+})| \geq n\}\bigr) =   (1-e^{-hn})^2\bP_{p}\bigl(\sS_{e,n}\bigr)
\end{multline*}
(Note that the first inequality is not an equality since $E(K_{e^-}) \cap E(K_{e^+})  \neq \emptyset$.) Applying \cref{thm:twoghostunimod}, it follows that if we set $h=c/n$ for $c>0$ then
\[
\bP_p(\sS_{e,n}) \leq (1-e^{-hn})^{-2}\bP_{p,h}(\sT_e)  \leq   33 c^{1/2} (1-e^{-c})^{-2} \deg(o) \left[\frac{1-p}{pn}\right]^{1/2}
\]
for every $n\geq 1$ and $c>0$.
The claim follows since $\inf_{c>0} c^{1/2} (1-e^{-c})^{-2} =1.873\ldots \leq 2$.
\end{proof}

\section{Proof of \cref{thm:polybound}}


We now apply \cref{cor:twoarmunimod} and the two-point function bound \eqref{eq:kappabound} to prove \cref{thm:polybound}.

\begin{proof}[Proof of \cref{thm:polybound}]
Let $G$ be a connected, locally finite graph, and let $0<p<p_c(G)$. Let $u,v$ be vertices of $G$, and let $\gamma$ be a simple path of length $k$ in $G$ starting at $u$ and ending at $v$. Let $\gamma_i $ denote the $i$th  edge that is traversed by $\gamma$. Let $\omega$ be Bernoulli bond percolation on $G$. For each $0\leq i \leq k$, let $\omega^i$ be obtained from $\omega$ by setting 
\[\omega^i(e) =\begin{cases} 1 & e \in \{\gamma_j : 1\leq j  \leq i\}\\
\omega(e) & e \notin \{\gamma_j : 1\leq j  \leq i\}.
\end{cases}
\]
Let $\sA_n(u,v)$ be the event that $u$ and $v$ are in distinct clusters of $\omega$ each of which touches at least $n$ edges. For each $1\leq i \leq k$, let $\sB_{n,i}(u,v,\gamma)$ that $u$ and $v$ are in distinct clusters of $\omega^{i-1}$ each of which touches at least $n$ edges and that $u$ and $v$ are connected in $\omega^i$. Since $u,v$ are clearly connected in $\omega^k$ and not connected in $\omega^0$ on the event $\sA_n(u,v)$, 
 and since the clusters of $u$ and $v$ are larger in $\omega^i$ than in $\omega^0$, we have that 
\begin{equation}
\label{eq:ABsetinclusion}
\sA_{n}(u,v) \subseteq \bigcup_{i=1}^k \sB_{n,i}(u,v,\gamma).
\end{equation}
Now, for each $1\leq i \leq k$ and $n\geq 1$ let $\sC_{n,i}(\gamma)$ be the event that the endpoints of $\gamma_i$ are in distinct clusters of $\omega$ each of which touches at least $n$ edges. Observe that $\sC_{n,i}(\gamma) \supseteq \sB_{n,i}(u,v,\gamma) \cap \{\omega(\gamma_j)=1 \text{ for every $1\leq j \leq i-1$}\}$. Since these two events are independent, we deduce that
\[
\bP_p\bigl(\sC_{n,i}(\gamma)) \geq p^{i-1} \bP_p\bigl( \sB_{n,i}(u,v,\gamma)),
\]
and hence by \eqref{eq:ABsetinclusion} that
\begin{equation*}
\bP_p(\sA_n(u,v)) \leq \sum_{i=1}^k p^{-i+1} \bP_p\bigl(\sC_{n,i}(\gamma))
\end{equation*}
On the other hand, since $p<p_c(G)$ we  have that $
\bP_p\bigl(\sC_{n,i}(\gamma) \bigr) \leq \sup_{e\in E} \bP_p\bigl(\sS_{e,n}\bigr)$ for every $1\leq i \leq k$, and it follows that
\begin{equation}
p^k\bP_p\bigl(\sA_n(u,v)\bigr) \leq \sum_{i=1}^k p^{k-i+1} \sup_{e\in E} \bP_p\bigl(\sS_{e,n}\bigr) \leq \frac{p}{(1-p)} \sup_{e\in E} \bP_p\bigl(\sS_{e,n}\bigr)
\end{equation}
for every $0<p<p_c(G)$ and every $u,v\in V$ with $d(u,v)\leq k$.

Now suppose that $G$ is transitive, unimodular, and has exponential volume growth. Let $P_p(n)=\bP_p(|E(K_v)|\geq n)$, which does not depend on the choice of $v$ by transitivity, and let $Q_p(n)=\sup_{e\in E} \bP_p(\sS_{e,n})$. For each $u,v\in V$ we have the bound 
\[\bP_p(\sA_n(u,v)) \geq \bP_p(|K_u| \geq n,|K_v|\geq n) - \tau_p(u,v) \geq P_p(n)^2 - \tau_p(u,v),\]
where the second inequality follows by transitivity and the Harris-FKG inequality. Thus, if we take $u,v$ to be a pair of vertices with $d(u,v)\leq k$ minimizing $\tau_p(u,v)$ and take $\gamma$ to be a geodesic between $u$ and $v$, we obtain that, by \eqref{eq:kappabound},
\[
p^k \left[ P_p(n)^2 -\gr^{-k} \right]  \leq p^k \left[ P_p(n)^2 -\kappa_p(k)\right] \leq \frac{p}{1-p}Q_p(n)
\]
for every $0<p<p_c(G)$, $n\geq 1$ and $k\geq 1$.
Setting 
$k= \left\lceil \frac{-\log \frac{1}{2}P_p(n)^2}{\log \gr}\right\rceil$ gives $\gr^{-k} \leq \frac{1}{2}P_p(n)^2$ and $p^k \geq p [\frac{1}{2}P_p(n)^2]^{-\log p/\log \gr}$
and rearranging 
we obtain that
\begin{equation}
\label{eq:calculation}
P_p(n)  \leq \sqrt{2}\left[ \frac{1}{1-p} Q_p(n)\right]^{1/2\alpha_p}
\end{equation}
for every $0<p<p_c(G)$ and $n\geq 1$, where $\alpha_p = 1 - \log p/\log \gr$. 
%
%
Thus, it follows from \eqref{eq:calculation} and \cref{cor:twoarmunimod} that
\[
P_p(n)\leq \sqrt{2} \left( 66 d \left[ \frac{1}{(1-p)n} \right]^{1/2} \right)^{1/2\alpha_p}.
\] 
Since $P_p(n)$ is a continuous function of $p$ for each $n\geq 1$ (indeed, it depends on only finitely many edges and is therefore a polynomial),  we may take the limit as $p\uparrow p_c$ to obtain that
\[
P_{p_c}(n)\leq \sqrt{2} \left( 66 d \left[ \frac{{1}}{(1-{p_c})n} \right]^{1/2} \right)^{1/2\alpha_{p_c}} \leq \frac{5 d^{1/4}}{(1-\gr^{-1})^{1/8}} n^{-1/4\alpha_{p_c}} \leq \frac{5 d^{1/4}}{(1-\gr^{-1})^{1/8}} n^{-\beta}
\] 
where $\beta = 1/(4+4 \log(d-1) /\log \gr)$. We have used that $p_c \leq \gr^{-1}$ and hence that
$\alpha_{p_c} \geq 2$ in the first inequality and similarly that $p_c \geq 1/(d-1)$ and hence that 
 $\beta \leq 1/4\alpha_{p_c}$ in the final inequality.
\end{proof}


%

\section{Proof of \cref{thm:locality}}

In this section we prove \cref{thm:locality}. The case in which all the graphs $G$ and $(G_n)_{n\geq 1}$ are unimodular is an immediate consequence of the following corollary of \cref{thm:polybound}.
Given two transitive graphs $G_1$ and $G_2$, we write $R(G_1,G_2)$ for the maximal radius $R$ such that whenever $o_1$ is a vertex of $G_1$ and $o_2$ is a vertex of $G_2$, there exists a graph isomorphism from the ball for radius $R$ around $o_1$ in $G_1$ to the ball of radius $R$ around $o_2$ in $G_2$ sending $o_1$ to $o_2$. 

\begin{corollary}
\label{cor:rateofconvergence}
For every $M<\infty$ and $g>1$, there exist positive constants $C(M,g)$ and $\delta(M,g)$ such that for every pair of  unimodular transitive graphs $G_1$ and $G_2$ with degrees at most $M$ and $\gr(G_1),\gr(G_2)\geq g$, we have that
\[
|p_c(G_1)-p_c(G_2)| \leq C R(G_1,G_2)^{-\delta}.
\]
\end{corollary}

Our proof will apply the following \emph{mean-field lower bound}: If $G$ is a connected, locally finite, transitive graph, $o$ is arbitrarily chosen root vertex of $G$, and $\theta_G(p)$ denotes the probability that $o$ is in an infinite cluster of $\omega_p$, then 
\begin{equation}
\label{eq:sharpness}
\theta_G(p) \geq \frac{p-p_c}{p(1-p_c)}
\end{equation}
for every $p_c<p\leq 1$.
The first inequality of this form was proven under more restrictive hypotheses by Chayes and Chayes \cite{MR864316}, with a more general version proven by Aizenman and Barsky \cite{aizenman1987sharpness}. The precise inequality \eqref{eq:sharpness} was proven by Duminil-Copin and Tassion \cite{duminil2015new}; it is a little stronger than the earlier results and also has a much simpler proof. (Note that one can prove locality from \cref{thm:polybound} without using this bound, but would not then get a quantitative estimate on the rate of convergence as we do here.)

\begin{proof}
We may assume that $R=R(G_1,G_2)\geq 1$, since the claim is trivial otherwise. In this case $G_1$ and $G_2$ must both have the same degree, which we denote by $d\leq M$.
Suppose without loss of generality that $p_c(G_1) > p_c(G_2)$. Let $o_1$ and $o_2$ be arbitrarily chosen root vertices of $G_1$ and $G_2$ respectively. Let $P_{p}^{G_1}(n)$ be the probability that the cluster of $o_1$ touches at least $n$ edges. Observe that we can clearly couple the percolation configurations on $G_1$ and $G_2$ so that if the cluster of $o_2$ is infinite then the cluster of $o_1$ must have diameter at least $R$, and must therefore contain at least $R$ vertices and touch at least $\left\lceil \frac{d}{2} R \right\rceil$ edges. Thus, by \eqref{eq:sharpness} and \cref{thm:polybound} we have that
\[
\frac{p-p_c(G_2)}{p(1-p_c(G_2))} \leq \theta_{G_2}(p) \leq P_p^{G_1}\left(\left\lceil \frac{d}{2} R \right\rceil\right) \leq C' \left[\frac{2}{dR}\right]^\delta
\]
for every $p_c(G_2) \leq p \leq 1$, 
where $C'=C'(M,g)$ and $\delta(M,g)$ are the constants from \cref{thm:polybound}. We conclude by setting $p=p_c(G_1)$ and rearranging. 
\end{proof}

\begin{remark}
\label{remark:rateofconvergence}
Explicitly, we obtain that if $p_1=p_c(G_1)\geq p_c(G_2)=p_2$ then
\[
|p_1-p_2|\leq p_1 (1-p_2)\sqrt{2} \left( 66 d  \left[ \frac{{2}}{(1-p_1) d R} \right]^{1/2} \right)^{1/2\alpha} \leq \frac{5 d^{1/8}p_1}{(1-\gr(G_1)^{-1})^{1/8}} R^{-\beta},
\]
where $\alpha=1- \log p_1/\log \gr(G_1)$ and $\beta = 1/4(1+\log(d-1)/\log \gr(G_1))$.
\end{remark}

\subsection{The nonunimodular case}

It remains to treat the case in which some of the graphs $(G_n)_{n\geq 1}$ are nonunimodular. This will be done by applying the results of \cite{Hutchcroftnonunimodularperc}, which give strong control of percolation on nonunimodular transitive graphs. In order to apply these results to prove locality, we will need to establish some relevant continuity properties of the modular function. We begin with some background on stationary random graphs, which we will then apply to study  the local structure of nonunimodular transitive graphs. Essentially, the purpose of the following discussion will be to give an alternative definition of the modular function, adapted from \cite{BC2011}, that will make its continuity properties more apparent. 

Recall that a \textbf{rooted graph} $(g,x)$ is a connected, locally finite graph $g$ together with a choice of distinguished root vertex $x$. A graph isomorphism between two rooted graphs is a rooted graph isomorphism if it preserves the root. Let $\fG_\bullet$ denote the space of isomorphism classes of rooted graphs. This space carries a natural topology, called the \textbf{local topology}, which is induced by the metric
\[
d_\mathrm{loc}\bigl((g_1,x_1),(g_2,x_2)\bigr) = R\bigl((g_1,x_1),(g_2,x_2)\bigr)^{-1}
\]
where, similarly to above, $R$ is the maximum radius such that the balls of radius $R$ around $x_1$ and $x_2$ in $g_1$ and $g_2$ respectively are isomorphic as rooted graphs. This is exactly the topology with which our graph limits were taken with respect to in \cref{conj:locality,thm:locality}. Similarly, we write $\fG_\bullet^\R$ for the space of rooted graphs $(g,x)$ that are equipped with a labeling $m$ of their oriented edges by real numbers. (This notation is nonstandard.) The space $\fG_\bullet^\R$ can also be equipped with a natural variant of the local topology, see e.g.\ \cite{AL07,CurienNotes} for details. We call a random variable taking values in $\fG_\bullet$ or $\fG_\bullet^\R$ a \textbf{random rooted graph} or \textbf{random rooted oriented-edge-labelled graph} as appropriate. We also define $\fG_{\bullet\bullet}$ to be the space of (isomorphism classes of) \textbf{doubly-rooted graphs}, that is, graphs with a distinguished ordered pair of root vertices, and define $\fG_{\bullet\bullet}^\R$ similarly. Again, these spaces carry natural variants of the local topology.

Let $(G,\rho)$ be a random rooted graph with law $\mu$. We say
that $(G,\rho)$ is \textbf{stationary} if it has the same distribution as $(G,X_1)$, where $X_1$ is the endpoint of the random oriented edge $\eta$ that is chosen uniformly at random from the set of oriented edges emanating from $\rho$. In particular, if $G$ is transitive and $o$ is an arbitrarily chosen root vertex of $o$ then $(G,o)$ is a stationary random rooted graph. Given a stationary random graph $(G,\rho)$ with law $\mu$, let $\mu_\rightarrow$ and $\mu_\leftarrow$ denote the laws of the random doubly-rooted graphs $(G,\rho,X_1)$ and $(G,X_1,\rho)$ respectively. It is shown in \cite{BC2011} that if $G$ has degrees bounded by $M$ a.s., then the measures $\mu_\rightarrow$ and $\mu_\leftarrow$ are absolutely continuous and that their Radon-Nikodym derivative $\frac{\dif \mu_\leftarrow}{\dif \mu_\rightarrow}$ satisfies
\begin{equation}
\label{eq:RNbound}
M^{-1} \leq \frac{\dif \mu_\leftarrow}{\dif \mu_\rightarrow}(g,x,y) \leq M
\end{equation}
for $\mu_\leftarrow$-a.e.\ doubly-rooted graph $(g,x,y)\in \fG_{\bullet\bullet}$. Moreover, it follows from \cite[Lemmas 4.1 and 4.2]{BC2011}  that for $\mu$-a.e.\ $(g,x)\in \fG_\bullet$, every $n\geq 1$ and every cycle $u_0 \sim u_1 \sim \cdots \sim u_n = u_0$ in $g$, we have that
\begin{equation}
\label{eq:RadonNikodymCocycle}
\prod_{i=0}^n \frac{\dif \mu_\leftarrow}{\dif \mu_\rightarrow}(g,u_i,u_{i+1})=1.
\end{equation}
Thus, we may define a function $\Delta_\mu(g,u,v)$ for $\mu$-a.e.\ $(g,x)\in \fG_\bullet$ and every two vertices $u,v$ of $g$ by
\begin{equation}
\label{eq:ModularRNdef}
\Delta_\mu(g,u,v)= \prod_{i=0}^n \frac{\dif \mu_\leftarrow}{\dif \mu_\rightarrow}(g,u_i,u_{i+1}),
\end{equation}
where $u=u_0\sim u_1 \sim \cdots \sim u_n =v$ is a path from $u$ to $v$ in $g$. The equality \eqref{eq:RadonNikodymCocycle} implies that the definition of $\Delta_\mu(g,u,v)$ is independent of the choice of path. We call $\Delta_\mu$ the \textbf{modular function} of $\mu$. 
It follows from \cite[Lemma 2.1]{Hutchcroftnonunimodularperc} that if $\mu$ puts all of its mass on a single deterministic transitive graph $(G,o)$, then the modular function $\Delta_\mu$ coincides with the modular function $\Delta_G$ of the transitive graph $G$, defined in \cref{section:unimodbackground}, in the sense that
\[
\Delta_\mu(g,u,v) = \mathbbm{1}(g \text{ isomorphic to } G) \Delta_G\bigl(\phi(u),\phi(v)\bigr) 
\]
for $\mu$-a.e.\ $(g,x)$ and all vertices $u,v$ of $g$, where $\phi$ is some isomorphism $g\to G$. (This is well defined as a function on $\fG_{\bullet\bullet}$ since the choice of $\phi$ does not affect the value obtained.)

Let $(G,\rho)$ be a stationary random rooted graph with law $\mu$. Then we can define a labeling of the oriented edges of $G$ by 
\[
m(e) = \Delta_\mu(G,e^-,e^+),
\] 
which is well-defined for every oriented edge $e$ of $G$ a.s.\ by \cite[Lemma 4.1]{BC2011}. Moreover, it is easily verified that $(G,\rho,m)$ is a stationary random rooted oriented-edge-labeled graph. 

\begin{prop}
\label{prop:modular}
Suppose that $((G_n,\rho_n))_{n\geq1}$ is a sequence of stationary random rooted graphs with laws $(\mu_n)_{n\geq1}$ converging in distribution to a stationary random rooted graph $(G,\rho)$ with law $\mu$, and suppose that there exists a constant $M$ such that all the graphs $(G_n)_{n\geq 1}$ and $G$ have degrees bounded by $M$ almost surely. Let $m_n$ and $m$ be the oriented edge-labellings of $G_n$ and $G$ that are defined in terms of the modular functions of $\mu_n$ and $\mu$ as above. Then $(G_n,\rho_n,m_n)$ converges locally to $(G,\rho,m)$.
\end{prop}

\begin{proof}
Let $\tilde \mu_n$ be the law of $(G_n,\rho_n,m_n)$ for each $n\geq 1$ and let $\tilde \mu$ be the law of $(G,\rho,m)$. It follows from \eqref{eq:RNbound} and the bounded degree assumption that the sequence of measures $(\tilde \mu_n)_{n\geq 1}$ is tight, and since $\fG_\bullet^\R$ is a Polish space \cite[Theorem 2]{CurienNotes} there exists a subsequence $\sigma(n)$ and a measure $\mu'$ such that $\mu_{\sigma(n)}\to\mu'$ weakly as $n\to\infty$. It suffices to prove that $\mu'=\tilde \mu$. Without loss of generality we may assume that $\sigma(n)=n$. Let $(G',\rho',m')$ a random variable with law $\mu'$. It is clear that $(G',\rho')$ has distribution $\mu$. Thus, by \eqref{eq:RadonNikodymCocycle}, \eqref{eq:ModularRNdef}, and \cite[Lemma 4.1]{BC2011}, to complete the proof it suffices to prove that if $\eta'$ is chosen uniformly from the set of oriented edges of $G'$ emanating from $\rho'$ and $X_1'$ is the other endpoint of $\eta'$ then
\begin{equation}
\label{eq:thingtocheck}
m'(\eta') = \frac{\dif\mu_\leftarrow}{\dif \mu_\rightarrow}(G',\rho',X_1')
\end{equation}
almost surely. To this end, let $F:\fG_{\bullet\bullet}\to \R$ be continuous and bounded. For each $n\geq 1$, let $\eta_n$ be  chosen uniformly from the set of oriented edges of $G_n$ emanating from $\rho_n$, and let $X_{1,n}$ be its endpoint. Then by definition of the Radon-Nikodym derivative we have that
\[
\E\left[ F(G_n,X_{1,n},\rho_n) \right] = \E\left[ F(G_n,\rho_n,X_{1,n}) \frac{\dif\mu_{n,\leftarrow}}{\dif \mu_{n,\rightarrow}}(G_n,\rho_n,X_{1,n})\right]
=
\E\left[ F(G_n,\rho_n,X_{1,n}) m_n(\eta_n)\right]
\]
for every $n\geq 1$, and taking $n\to\infty$ we obtain that, since $(G_n,\rho_n,m_n)$ converges weakly to $(G',\rho',m')$,
\[
\E\left[ F(G',X_{1}',\rho') \right] = \E\left[ F(G',\rho',X_{1}') m'(\eta')\right].
\]
Since $F$ was arbitrary this implies \eqref{eq:thingtocheck}, completing the proof.
\end{proof}

Now suppose that $G$ is a transitive graph. 
 Then it follows from the definition of the modular function given in \cref{section:unimodbackground} that if $G$ has degree $d$ then
\[
\Delta_{G}(u,v) \in \Bigl\{ \frac{a}{b} : a,b \in \{1,2,\ldots,d\} \Bigr\}
\]
for every every pair of neighbouring vertices $u,v$ in $G$. Since this set is finite, \cref{prop:modular} has the following immediate corollary.

\begin{corollary}
\label{cor:modularcontinuity}
Let $(G_n)_{n\geq1}$ be a sequence of transitive graphs converging locally to a transitive graph $G$ and let $(o_n)_{n\geq 1}$ and $o$ be arbitrarily chosen root vertices of $(G_n)_{n\geq 1}$ and $G$ respectively. Then 
for every $r\geq 1$ there exists $N<\infty$ such that for every $n\geq N$, there exists an isomorphism $\phi_n$ from the ball of radius $r$ around $o_n$ in $G_n$ to the ball of radius $r$ around $o$ in $G$ that sends $o_n$ to $o$ and satisfies
\begin{equation}
\Delta_{G_n}(u,v) = \Delta_G(\phi_n(u),\phi_n(v))
\end{equation}
for every $u,v$ in the ball of radius $r$ around $o_n$ in $G_n$.
\end{corollary}

Note that this property is not at all obvious from the algebraic definition of $\Delta$ given in \cref{section:unimodbackground}! A further immediate corollary is as follows, which implies that the set of unimodular transitive graphs is both closed and open as a subset of the space of all transitive graphs.

\begin{corollary}
\label{cor:clopen}
Let $(G_n)_{n\geq1}$ be a sequence of transitive graphs converging locally to a transitive graph $G$. If $G_n$ is unimodular for infinitely many $n$ then $G$ is unimodular, while if $G_n$ is nonunimodular for infinitely many $n$ then $G$ is nonunimodular.
\end{corollary}

We now combine \cref{cor:modularcontinuity} with the results of \cite{Hutchcroftnonunimodularperc} to prove the following theorem.

\begin{thm}
\label{thm:localitynonunimod}
Let $(G_n)_{n\geq 1}$ be a sequence of transitive graphs converging locally to a nonunimodular transitive graph $G$. Then $p_c(G_n)\to p_c(G)$ as $n\to\infty$.
\end{thm}

We begin by recalling the relevant results from \cite{Hutchcroftnonunimodularperc}. 
Let $G$ be a nonunimodular transitive graph with modular function $\Delta$. 
For each $t\in \R$ we define the \textbf{upper half-space} $H_t = \{v \in V: \log \Delta(o,v) \geq t \}$. For each $t\geq 0$, let 
\[A_p(t) = A_p^G(t) := \bP_p(o \xleftrightarrow{H_0} H_t)\]
be the probability that $o$ is connected to $H_t$ by an open path using only edges both endpoints of which are contained in $H_0$. Similarly, for each $t\geq 0$ and $r\geq 0$ let $A_p^G(t,r)$ be the probability that $o$ is connected to $H_t$ by an open path of length at most $r$ using only edges both endpoints of which are contained in $H_0$. 
The proof of \cite[Lemma 5.2]{Hutchcroftnonunimodularperc} yields that $A_p(t)$ satisfies the supermultiplicative inequality $A_p(t+s)\geq A_p(t)A_p(s)$ for every $t,s\geq 0$. Applying Fekete's Lemma, it follows that the limit
\[
\alpha_p(G):=-\lim_{t\to \infty} \frac{1}{t} \log A_p(t) = -\sup_{t\geq 1} \frac{1}{t} \log A_p(t)
\]
is well-defined. (We caution the reader not to confuse this use of the letter $\alpha$ with the notation used in Remarks \ref{remark:exponent} and \ref{remark:rateofconvergence}.) It follows from \cite[Theorem 1.8 and Corollary 5.10]{Hutchcroftnonunimodularperc} that $\alpha_{p_c}(G)=1$ for every transitive nonunimodular graph $G$, so that in particular
\begin{equation}
\label{eq:Feketebound}
A_p(t) \leq e^{-t} \qquad \text{ for every $p\leq p_c$ and $t\geq 0$.}
\end{equation}
(In \cite{Hutchcroftnonunimodularperc} the more complicated quasi-transitive case is treated; both the proof and the bounds obtained can be simplified substantially in the transitive case.) Moreover, it follows from \cite[Theorem 2.38]{grimmett2010percolation} that $\alpha_p(G)$ is a strictly decreasing function of $p$ when it is positive.

\begin{proof}[Proof of \cref{thm:localitynonunimod}]
As discussed in the introduction, the estimate $\liminf_{n\to\infty} p_c(G_n) \geq p_c(G)$ follows from general considerations (see \cite[Page 4]{duminil2015new} and \cite[\S14.2]{Pete}), and so it suffices to prove that $\limsup_{n\to\infty} p_c(G_n) \leq p_c(G)$. 
It follows from \cref{cor:clopen} that $G_n$ is nonunimodular for every sufficiently large $n$, and so we may suppose without loss of generality that $G_n$ is nonunimodular for every $n\geq 1$. It follows from \cref{cor:modularcontinuity} that 
$\lim_{n\to\infty} A_p^{G_n}(t,r) = A_p^G(t,r)$ for every fixed $p\in [0,1]$ and $t,r\geq 0$, and we deduce that
\[
\limsup_{n\to\infty}  A_p^{G_n}(t) = \limsup_{n\to\infty} \sup_{r\geq 1} A_p^{G_n}(t,r) \geq \sup_{r\geq 1} \limsup_{n\to\infty}  A_p^{G_n}(t,r) = \sup_{r\geq 1}   A_p^{G}(t,r)=A_p^G(t)
\]
for every $p\in [0,1]$ and $t\geq 0$. Thus, we obtain that
\begin{multline*}
\liminf_{n\to\infty} \alpha_p(G_n) = - \limsup_{n\to\infty} \sup_{t\geq 1} \frac{1}{t} \log A^{G_n}_p(t) \leq - \sup_{t\geq 1}  \limsup_{n\to\infty} \frac{1}{t} \log A^{G_n}_p(t)\\ \leq - \sup_{t\geq 1}   \frac{1}{t} \log A^{G}_p(t) = \alpha_p(G).
\end{multline*}
If $p>p_c(G)$ then $\alpha_p(G)<1$ as discussed above, so that $ \alpha_p(G_n)<1$ for infinitely many $n$. Thus, it follows from the results of \cite{Hutchcroftnonunimodularperc} that $p \geq \limsup_{n\to\infty}p_c(G_n)$ as required.
\end{proof}

\begin{proof}[Proof of \cref{thm:locality}]
This is an immediate consequence of \cref{cor:rateofconvergence,cor:clopen,thm:localitynonunimod}.
\end{proof}

\section{Closing remarks}

\begin{remark}
\label{remark:improvedtwoghost}
Various improvements to the inequalities of \cref{thm:twoghostunimod,cor:twoarmunimod} are possible, which we discuss briefly now.
\begin{enumerate}[leftmargin=*]
\item The proof of \cref{thm:twoghostunimod,cor:twoarmunimod} also applies \emph{mutatis mutandis} to \emph{reversible random rooted graphs}. See e.g.\ \cite{AL07,BC2011,CurienNotes} for introductions to this notion. Running the proof in this setting, we obtain that if $(G,\rho)$ is a reversible random rooted graph and $\eta$ is chosen uniformly at random from the set of oriented edges of $G$ emanating from $\rho$, then 
\begin{equation}
\label{eq:reversible}
\P_{p,h}(\sT_\eta) \leq 33 \left[\frac{1-p}{p} h\right]^{1/2}
\quad \text{ and } \quad \P_{p}(\sS_{\eta,n}) \leq 66 \left[\frac{1-p}{pn} \right]^{1/2}
\end{equation}
for every $p\in (0,1]$, $h> 0$, and $n\geq 1$. Similar results can be deduced for unimodular random rooted graphs of finite expected degree by applying the usual dictionary to translate between unimodularity and reversibility, see e.g.\ \cite{BC2011}. In particular, one can apply these results to obtain analogues of \cref{thm:twoghostunimod,cor:twoarmunimod} for the random cluster model, which can be viewed as Bernoulli bond percolation in a random environment. We plan to apply these bounds in forthcoming work. 

\item When $G$ is an amenable transitive graph, one can apply a standard exhaustion argument to remove the condition that the clusters are finite from \cref{thm:twoghostunimod,cor:twoarmunimod}, so that we obtain a new proof of uniqueness from our methods that gives better quantitative control of two-arm events than the classical proofs of \cite{MR901151,burton1989density}. Indeed, since $G$ is amenable there exists a sequence of a.s.\ finite reversible random rooted graphs $(G_n,\rho_n)_{n\geq 1}$ that are \emph{subgraphs} of $G$ and that almost surely converge to $G$ locally as $n\to\infty$. (To construct such a sequence, simply take a F{\o}lner sequence for $G$ and choose a random root from each F{\o}lner set according to the stationary measure on the set.) Applying \eqref{eq:reversible} to each of the graphs $(G_n,\rho_n)$ and then sending $n\to\infty$, one obtains that percolation on $G$ satisfies
\begin{equation}
\label{eq:reversible2}
\P_{p,h}(\sT_\eta^\infty) \leq 33 \left[\frac{1-p}{p} h\right]^{1/2}
\quad \text{ and } \quad \P_{p}(\sS_{\eta,n}^\infty) \leq 66 \left[\frac{1-p}{pn} \right]^{1/2}
\end{equation}
for every $p\in (0,1]$, $h>0$ and $n\geq 1$, where $\sT_e^\infty$ and $\sS^\infty_{e,n}$ denote the events that $e$ is closed and that the endpoints of $e$ are in distinct (not necessarily finite) clusters each of which either touches some green edge or touches at least $n$ edges, respectively. Similar inequalities hold  for invariantly amenable unimodular random rooted graphs of finite expected degree.

 Note that is is crucially important that the graphs $(G_n)_{n\geq 1}$ are \emph{subgraphs} of $G$ for this argument to work. This allows us to couple percolation on each of the graphs $(G_n,\rho_n)$ with percolation on $(G,\rho)$ in such a way that the clusters of the endpoints of $\rho$ are distinct in $G$ if and only if they are distinct in $G_n$ for every $n\geq 1$; this property is needed to deduce \eqref{eq:reversible2} by taking limits in \eqref{eq:reversible}. Indeed, one can also approximate, say, a $3$-regular tree by finite reversible random graphs, but the argument does not work in this case since the two clusters at either end of an edge could merge in the finite graphs far away from the edge without merging in the limiting tree. Of course, we also know that the argument cannot work in this setting since we have infinitely many infinite clusters on the $3$-regular tree when $p>p_c=1/2$.

\item If one assumes that $G$ is unimodular and that $\bP_{p_c}(|E(K_o)|\geq n) \leq C_1n^{-\gamma}$ for some $\gamma < 1/2$ and $C_1<\infty$, then one can improve upon \eqref{eq:twobodyunimod} to obtain that
\begin{equation}
\label{eq:improved}
\bP_{p_c,h}(\sT_e) \leq  C_2h^{\gamma+1/2} \quad \text{ and } \quad  \bP_{p_c}(\sS_{e,n}) \leq  C_2n^{-\gamma-1/2}
\end{equation}
for some constant $C_2$ and every $h>0$ and $n\geq 1$.
This improves substantially on the bound that one obtains from the BK inequality, which gives $\bP_{p_c}(\sS_{e,n}) \leq C^2_1 n^{-2\gamma}$. 
To obtain such an improved bound, first note that in this case one can 
bound $\bE_{p_c}[T\wedge n] \leq C_3 n^{1-\gamma}$ for some constant $C_3$ instead of using the trivial bound $\bE_{p_c}[T\wedge n] \leq n$. Applying Cauchy-Schwarz we obtain that
\begin{multline*}
\sum_{k\geq 0} \frac{1-e^{-h2^{k}}}{2^k} \bE_{p_c}\left[ \max_{0\leq n \leq 2^{k+1}} |Z_n| \mathbbm{1}(T\geq 2^k) \right]
\\\leq 
2\sqrt{p_c(1-p_c)}\sum_{k\geq 0} \frac{1-e^{-h2^{k}}}{2^k} \bE_{p_c}[T\wedge 2^{k+1}]^{1/2}
\bP_{p_c}(T\geq 2^k)^{1/2}
\leq 
C_4\sum_{k\geq 0} \frac{1-e^{-h2^{k}}}{2^{(1+2\gamma)k/2}}
\end{multline*}
for some constant $C_4$ and every $h>0$. 
The rest of the proof proceeds similarly to before.
\end{enumerate}
\end{remark}

\begin{remark}
\label{remark:improvedpolybound}
Applying the bound \eqref{eq:improved} in place of \eqref{eq:twobodyunimod} in the proof of \cref{thm:polybound}, one can make iterative improvements to the exponent bound that the argument yields. By applying this iterative procedure an arbitrarily large number of times, one obtains the bound
\[
\bP_{p_c} \bigl(|E(K_o)|\geq n\bigr) \leq n^{-1/(4\alpha-2)+o(1)}
\]
that was mentioned in \cref{remark:exponent}.
\end{remark}

\begin{remark}
Applying the results of \cite{1709.10515} instead of those of \cite{Hutchcroftnonunimodularperc}, we obtain the following locality result for the self-avoiding walk connective constant. 
\begin{thm}
Let $(G_n)_{n\geq 1}$ be a sequence of transitive graphs converging locally to a nonunimodular transitive graph $G$. Then $\mu_c(G_n)\to \mu_c(G)$ as $n\to\infty$.
\end{thm}

\begin{proof}
The exponential rate of decay $\alpha_w(z)$ defined in eq.\ (2.4) of \cite{1709.10515} plays the role that $\alpha_p$ played in the proof of \cref{thm:localitynonunimod}. The equality $\alpha_w(z_c)=1$ follows from eq.\ (2.9) of \cite{1709.10515}, and is the direct analogue of the fact that $\alpha_{p_c}=1$ that we quoted from \cite{Hutchcroftnonunimodularperc}. Finally, the fact that $\alpha_w(z)$ is strictly decreasing when it is positive is given in \cite[Lemma 2.4]{1709.10515}. With these facts in hand, the proof proceeds identically to that of \cref{thm:localitynonunimod}. (Again, the use of the letter $\alpha$ in this proof should not be confused with the notation for the exponents appearing in Remarks \ref{remark:exponent} and \ref{remark:rateofconvergence}.)
\end{proof}

See \cite{1412.0150,1704.05884} for further results regarding locality of the connective constant.
\end{remark}


\subsection*{Acknowledgments}
We thank Jonathan Hermon for many helpful discussions, and for his careful reading of an earlier version of this manuscript. 
This paper also greatly benefited from discussions with Vincent Tassion on rewriting the Aizenman-Kesten-Newman uniqueness proof with martingale techniques, which took place at the Isaac Newton Institute during the RGM follow up workshop.
We also thank Itai Benjamini, Hugo Duminil-Copin, Gady Kozma, Russ Lyons, Sebastien Martineau, Asaf Nachmias, and an anonymous referee for comments on earlier versions of the paper.

 \setstretch{1}
 \footnotesize{
  \bibliographystyle{abbrv}
  \bibliography{unimodularthesis.bib}

\begin{thebibliography}{10}

\bibitem{aizenman1987sharpness}
M.~Aizenman and D.~J. Barsky.
\newblock Sharpness of the phase transition in percolation models.
\newblock {\em Communications in Mathematical Physics}, 108(3):489--526, 1987.

\bibitem{MR901151}
M.~Aizenman, H.~Kesten, and C.~M. Newman.
\newblock Uniqueness of the infinite cluster and continuity of connectivity
  functions for short and long range percolation.
\newblock {\em Comm. Math. Phys.}, 111(4):505--531, 1987.

\bibitem{AL07}
D.~Aldous and R.~Lyons.
\newblock Processes on unimodular random networks.
\newblock {\em Electron. J. Probab.}, 12:no. 54, 1454--1508, 2007.

\bibitem{benjamini2013euclidean}
I.~Benjamini.
\newblock Euclidean vs. graph metric.
\newblock In {\em Erd{\H{o}}s Centennial}, pages 35--57. Springer, 2013.

\bibitem{BC2011}
I.~Benjamini and N.~Curien.
\newblock Ergodic theory on stationary random graphs.
\newblock {\em Electron. J. Probab.}, 17:no. 93, 20, 2012.

\bibitem{BLPS99b}
I.~Benjamini, R.~Lyons, Y.~Peres, and O.~Schramm.
\newblock Critical percolation on any nonamenable group has no infinite
  clusters.
\newblock {\em Ann. Probab.}, 27(3):1347--1356, 1999.

\bibitem{MR2773031}
I.~Benjamini, A.~Nachmias, and Y.~Peres.
\newblock Is the critical percolation probability local?
\newblock {\em Probab. Theory Related Fields}, 149(1-2):261--269, 2011.

\bibitem{burton1989density}
R.~M. Burton and M.~Keane.
\newblock Density and uniqueness in percolation.
\newblock {\em Communications in mathematical physics}, 121(3):501--505, 1989.

\bibitem{MR3395466}
R.~Cerf.
\newblock A lower bound on the two-arms exponent for critical percolation on
  the lattice.
\newblock {\em Ann. Probab.}, 43(5):2458--2480, 2015.

\bibitem{MR864316}
J.~T. Chayes and L.~Chayes.
\newblock Inequality for the infinite-cluster density in {B}ernoulli
  percolation.
\newblock {\em Phys. Rev. Lett.}, 56(16):1619--1622, 1986.

\bibitem{CurienNotes}
N.~Curien.
\newblock Random graphs: the local convergence point of view.
\newblock 2017.
\newblock Unpublished lecture notes. Available at
  \url{https://www.math.u-psud.fr/~curien/cours/cours-RG-V3.pdf}.

\bibitem{1806.07733}
H.~Duminil-Copin, S.~Goswami, A.~Raoufi, F.~Severo, and A.~Yadin.
\newblock Existence of phase transition for percolation using the {G}aussian
  free field.
\newblock 2018.
\newblock arXiv:1806.07733.

\bibitem{duminil2015new}
H.~Duminil-Copin and V.~Tassion.
\newblock A new proof of the sharpness of the phase transition for {B}ernoulli
  percolation and the {I}sing model.
\newblock {\em Communications in Mathematical Physics}, pages 1--21, 2015.

\bibitem{1707.07626}
H.~Duminil-Copin and V.~Tassion.
\newblock A note on {S}chramm's locality conjecture for random-cluster models,
  2017.
\newblock arXiv:1707.07626.

\bibitem{MR929129}
A.~Gandolfi, G.~Grimmett, and L.~Russo.
\newblock On the uniqueness of the infinite cluster in the percolation model.
\newblock {\em Comm. Math. Phys.}, 114(4):549--552, 1988.

\bibitem{grimmett2010percolation}
G.~R. Grimmett.
\newblock Percolation (grundlehren der mathematischen wissenschaften).
\newblock 2010.

\bibitem{1412.0150}
G.~R. Grimmett and Z.~Li.
\newblock Locality of connective constants.
\newblock 2014.

\bibitem{1704.05884}
G.~R. Grimmett and Z.~Li.
\newblock Self-avoiding walks and connective constants, 2017.

\bibitem{MR1068308}
G.~R. Grimmett and J.~M. Marstrand.
\newblock The supercritical phase of percolation is well behaved.
\newblock {\em Proc. Roy. Soc. London Ser. A}, 430(1879):439--457, 1990.

\bibitem{heydenreich2015progress}
M.~Heydenreich and R.~van~der Hofstad.
\newblock Progress in high-dimensional percolation and random graphs.

\bibitem{Hutchcroft2016944}
T.~Hutchcroft.
\newblock Critical percolation on any quasi-transitive graph of exponential
  growth has no infinite clusters.
\newblock {\em Comptes Rendus Mathematique}, 354(9):944 -- 947, 2016.

\bibitem{Hutchcroftnonunimodularperc}
T.~Hutchcroft.
\newblock Non-uniqueness and mean-field criticality for percolation on
  nonunimodular transitive graphs.
\newblock 2017.
\newblock arXiv:1711.02590.

\bibitem{1709.10515}
T.~Hutchcroft.
\newblock Self-avoiding walk on nonunimodular transitive graphs.
\newblock {\em Annals of Probability}, 2017.
\newblock To appear. Preprint available at arXiv:1709.10515.

\bibitem{1804.10191}
T.~Hutchcroft.
\newblock Percolation on hyperbolic graphs.
\newblock {\em Geometric and Functional Analysis}, 2018.
\newblock To appear. Preprint available at arXiv:1804.10191.

\bibitem{LP:book}
R.~Lyons and Y.~Peres.
\newblock {\em Probability on Trees and Networks}.
\newblock Cambridge University Press, New York, 2016.
\newblock Available at \url{http://pages.iu.edu/~rdlyons/}.

\bibitem{MR3630298}
S.~Martineau and V.~Tassion.
\newblock Locality of percolation for {A}belian {C}ayley graphs.
\newblock {\em Ann. Probab.}, 45(2):1247--1277, 2017.

\bibitem{MR986363}
B.~Mohar and W.~Woess.
\newblock A survey on spectra of infinite graphs.
\newblock {\em Bull. London Math. Soc.}, 21(3):209--234, 1989.

\bibitem{MR3005730}
A.~Nachmias and Y.~Peres.
\newblock Non-amenable {C}ayley graphs of high girth have {$p_c<p_u$} and
  mean-field exponents.
\newblock {\em Electron. Commun. Probab.}, 17:no. 57, 8, 2012.

\bibitem{MR1756965}
I.~Pak and T.~Smirnova-Nagnibeda.
\newblock On non-uniqueness of percolation on nonamenable {C}ayley graphs.
\newblock {\em C. R. Acad. Sci. Paris S\'er. I Math.}, 330(6):495--500, 2000.

\bibitem{Pete}
G.~Pete.
\newblock Probability and geometry on groups.
\newblock http://www.math.bme.hu/~gabor/PGG.pdf, 2014.

\bibitem{MR1833805}
R.~H. Schonmann.
\newblock Multiplicity of phase transitions and mean-field criticality on
  highly non-amenable graphs.
\newblock {\em Comm. Math. Phys.}, 219(2):271--322, 2001.

\bibitem{MR1082868}
P.~M. Soardi and W.~Woess.
\newblock Amenability, unimodularity, and the spectral radius of random walks
  on infinite graphs.
\newblock {\em Math. Z.}, 205(3):471--486, 1990.

\bibitem{song2014locality}
H.~Song, K.-N. Xiang, and S.-C.-H. Zhu.
\newblock Locality of percolation critical probabilities: uniformly nonamenable
  case.
\newblock {\em arXiv preprint arXiv:1410.2453}, 2014.

\bibitem{timar2006percolation}
{\'A}.~Tim{\'a}r.
\newblock Percolation on nonunimodular transitive graphs.
\newblock {\em The Annals of Probability}, pages 2344--2364, 2006.

\bibitem{werner2007lectures}
W.~Werner.
\newblock Lectures on two-dimensional critical percolation.
\newblock {\em arXiv preprint arXiv:0710.0856}, 2007.

\end{thebibliography}
}
\end{document}